\documentclass[a4paper]{amsart}
\usepackage{graphicx}
\usepackage[colorlinks, linkcolor= blue, citecolor= red]{hyperref}
\usepackage{pdfsync}
\usepackage{color}

\usepackage[T1]{fontenc}
\usepackage[utf8]{inputenc} 

\usepackage[all]{xy}

\usepackage{amsmath, amssymb, amsfonts, amscd, amsthm, mathrsfs}

\renewcommand{\phi}{\varphi}
\newcommand{\C}{\mathbb{C}}
\newcommand{\z}{\mathbb{Z}}

\newcommand{\q}{\mathbb{Q}}

\newcommand{\rel}{\operatorname{rel}}
\newcommand{\stab}{\operatorname{Stab}}

\newtheoremstyle{pedro}{}{}{\itshape}{}{\sc}{~--}{ }{\thmname{#1}\thmnumber{ #2}\thmnote{ (#3)}}

\newtheoremstyle{pedrodef}{}{}{}{}{\sc}{~--}{ }{\thmname{#1}\thmnumber{ #2}\thmnote{ (#3)}}

\theoremstyle{pedro}
\newtheorem{lem}{Lemma}[section]

\newtheorem{thm}[lem]{Theorem}

\newtheorem{prop}[lem]{Proposition}

\newtheorem{coro}[lem]{Corollary}

\theoremstyle{remark}

\newtheorem{rmk}[lem]{Remark}

\theoremstyle{pedrodef}

\newtheorem{ex}[lem]{Example}

\title{Cayley graphs and automatic sequences}

\author{Pierre Guillot}

\address{
Universit\'{e} de Strasbourg \& CNRS\\
Institut de Recherche Math\'{e}matique Avanc\'{e}e\\
7~Rue Ren\'{e} Descartes\\
67084 Strasbourg, France}

\email{guillot@math.unistra.fr}

\let\oldtocsection=\tocsection
\let\oldtocsubsection=\tocsubsection
\let\oldtocsubsubsection=\tocsubsubsection

\renewcommand{\tocsection}[2]{\hspace{0em}\oldtocsection{#1}{#2}}
\renewcommand{\tocsubsection}[2]{\hspace{2em}\oldtocsubsection{#1}{#2}}
\renewcommand{\tocsubsubsection}[2]{\hspace{2em}\oldtocsubsubsection{#1}{#2}}

\numberwithin{equation}{section}

\begin{document}

\maketitle

\begin{abstract}
We study those automatic sequences which are produced by an automaton whose underlying graph is the Cayley graph of a finite group. For~$2$-automatic sequences, we find a characterization in terms of what we call {\em homogeneity}, and among homogeneous sequences, we single out those enjoying what we call {\em self-similarity}. It turns out that self-similar~$2$-automatic sequences (viewed up to a permutation of their alphabet) are in bijection with many interesting objects, for example dessins d'enfants (covers of the Riemann sphere with three points removed).

For any~$p$ we show that, in the case of an automatic sequence produced ``by a Cayley graph'', the group and indeed the automaton can be recovered canonically from the sequence.

Further, we show that a rational fraction may be associated to any automatic sequence. To compute this fraction explicitly, knowledge of a certain graph is required. We prove that for the sequences studied in the first part, the graph is simply the Cayley graph that we start from, and so calculations are possible.

We give applications to the study of the frequencies of letters.
\end{abstract}

\section{Introduction}

\subsection{Basic definitions}

Let~$p \ge 2$ be an integer, which in practice will often be a prime. A~$p$-{\em automaton} is, first and foremost, a directed graph on a finite set~$Q$, whose elements are called the {\em states} ; the following extra decoration is required : \begin{itemize}
\item There is a distinguished state, called the {\em initial} state.
\item The arrows are labeled with the integers~$i$ such that~$0 \le i < p$.
\item The vertices are labeled using a map~$\tau \colon Q \to \Delta $. Here~$\Delta $ is a finite set called the {\em alphabet}. Typically~$\tau $ will be surjective, but may very well fail to be injective.
\end{itemize}
Finally, the following property must be satisfied: out of each state (vertex), there is precisely one arrow labeled~$i$, for each~$0 \le i < p$. For example, here is a~$2$-automaton.

\begin{center}
\includegraphics[width=\textwidth]{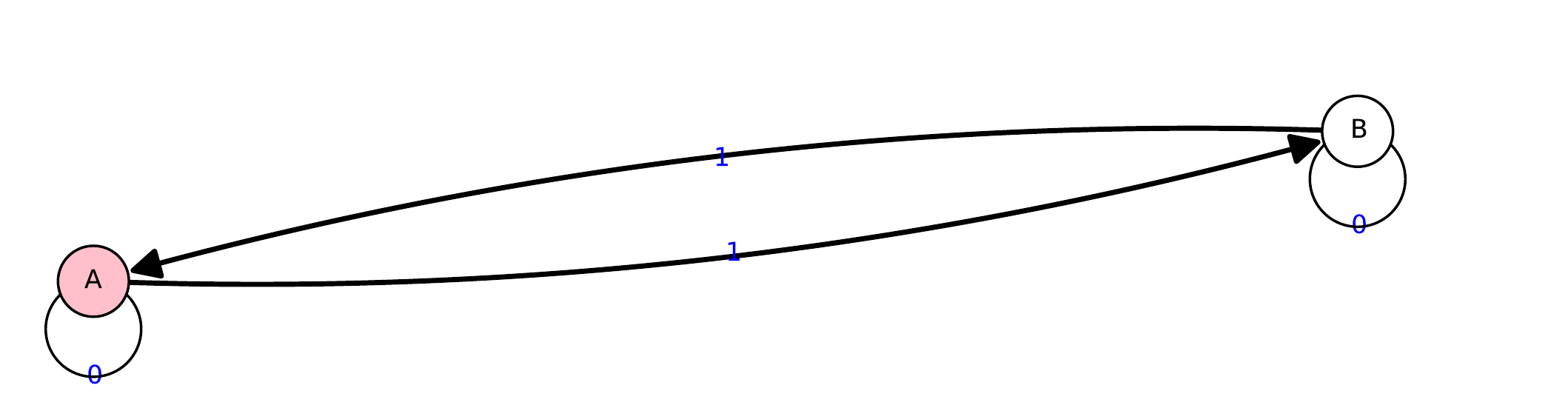} \\
{\em \footnotesize Fig.\ 1}
\end{center}
In this example the alphabet is~$\Delta = \{ A, B \}$ ; the initial state is the state bearing the label~$A$. On our pictures we usually depict the initial state in a darker colour.

This definition is equivalent to others in the literature, and it will serve our purposes well, at the cost of (mildly) surprising some readers.

An automaton produces a sequence~$(a_n)_{n \ge 1}$ of elements of~$\Delta $ using the following recipe. Read the digits of~$n$ written in base~$p$, from right to left, and follow the corresponding arrows in the graph, starting from the initial state ; the label of the state reached in this fashion is~$a_n$. Using the automaton from Figure~$1$, one obtains the celebrated Thue-Morse sequence, starting with 
\[ BBABAABBAABABBABAABABBAABBABAABB \ldots  \] 
Many readers familiar with the Thue-Morse sequence will complain that there is an~$A$ missing at the beginning. However, we chose to define the sequence associated to an automaton for~$n \ge 1$ (rather than~$n \ge 0$), and we make no apology for this unorthodox decision: the results of the present paper will irremediably fail to hold for sequences defined from~$0$.

A sequence of elements of~$\Delta $ which is the output of at least one~$p$-automaton is called {\em $p$-automatic}. There is a vast literature on automatic sequences (see~\cite{allouche}), and most results do not depend on whether the sequences start from~$0$ or~$1$.

The automaton on Figure~$2$ appeared in~\cite{montreal} and~\cite{notices}.

\begin{center}
\includegraphics[width=.75\textwidth]{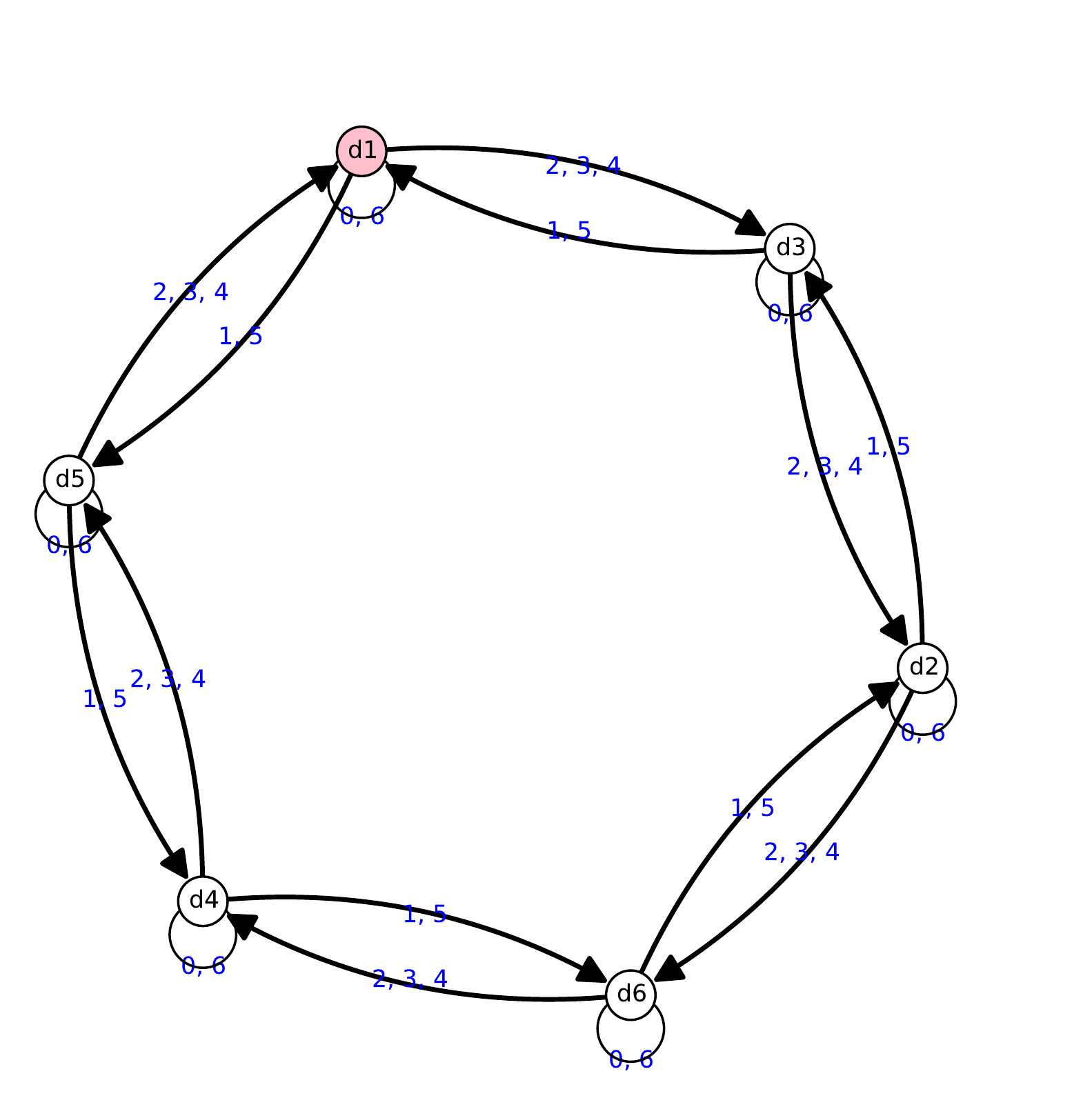}\\
{\em \footnotesize Fig.\ 2}
\end{center}
On this picture we have written several labels next to a given arrow as a way of saving space. There are in fact, for example, three arrows going from~$d_1$ to~$d_3$, with labels~$2$, $3$ and~$4$ respectively.

The sequence produced by this automaton is~$(A(n) \mod 7)_{n \ge 1}$ where~$A(n)$ is the Ap\'ery number 
\[ A(n) = \sum_{k=0}^n {n \choose k}^2 {n+k \choose k}^2 \, .   \]
(More precisely if the output for a given~$n$ is~$d_i$, then~$A(n) \mod 7$ is~$i$. We chose the alphabet to be~$\Delta = \{ d_1, \ldots, d_6 \}$ rather than $\{ 1, \ldots, 6 \}$ to emphasize that~$d_i$ is a formal symbol and that no arithmetic is performed with these outputs in the sequel.)

In~\cite{notices}, Rowland comments that this automaton is ``particularly symmetric''. In this paper we investigate those sequences which are produced by automata with a lot of symmetry.

One may think of an automaton as a particular kind of computing machine, which outputs~$a_n$ when fed~$n$. If one adopts this point of view, the question of symmetry is not natural: one hardly cares for symmetric computers or computer programs. However, it turns out that the {\em answer} to this surprising question is very simple and satisfying (especially for~$p=2$).

First, we need to define what we mean by symmetry. Suppose~$G$ is a group with distinguished generators~$t_0, t_1, \ldots, t_{p-1}$. Its {\em Cayley graph} is built as follows: the set of vertices is~$G$ itself, and there is one arrow with the label~$i$ placed between~$g$ and~$gt_i$ (for all~$g\in G$ and~$0 \le i <p$). 

Both our examples above involve Cayley graphs. If we take the symmetric group~$S_2$, generated by~$t_0 = I$ (the identity permutation) and~$t_1= (1, 2)$, we get precisely the graph underlying Figure~$1$. As for Figure~$2$, consider the permutations
\[ t_0 = t_6 = I \, ,  \quad t_1 = t_5 = (1, 5, 4, 6, 2, 3) \, , \quad  t_2 = t_3 = t_4 = (1, 3, 2, 6, 4, 5) \, ,   \]
and the group~$G$ that they generate (which is cyclic of order~$6$); what we have on the figure is precisely the corresponding Cayley graph.

Note that a Cayley graph always has a distinguished vertex, namely the identity of~$G$. Also, the characteristic property for automata is satisfied {\em as well as it ``dual''}, namely: at each vertex there is exactly one arrow with the label~$i$ going out, and also exactly one arrow with the label~$i$ coming in. However to turn a Cayley graph into an automaton, one needs to provide a map~$\tau \colon G \to \Delta $.

\subsection{Two types of sequences with symmetry}

We propose to give necessary and sufficient conditions for an automatic sequence to be produced by an automaton whose underlying graph is a Cayley graph. These have a vertex-transitive group of automorphisms, and are thus very symmetric objects indeed.

Here is some notation to formulate these conditions. Let the integer~$p$ be fixed throughout. For~$i, j \ge 0$ with~$j < p^i$, we define the subsequence~$a ^{(i,j)}$ of the sequence~$a$ by 
\[ a ^{(i,j)}_n = a_{p^in + j}    \]
for~$n \ge 1$ (not~$n \ge 0$ !). We let 
\[ N(a) = \{ a ^{(i,j)} : i,j \ge 0, ~ j < p^i \}  \]
be the set of such sequences. (As we shall recall below, it is well-known that~$N(a)$ is finite if and only if~$a$ is $p$-automatic.) 

Further, a {\em relation} for the sequence~$a$ is a pair~$(i,j)$ such that~$a ^{(i,j)} = a$ ; we exclude~$(i,j) = (0,0)$, which we do not consider as a relation. We write~$\rel(a)$ for the set of all relations of~$a$.

We will say that~$a$ is {\em homogeneous} if 
\[ \forall u \in N(a) \, , \quad N(u) = N(a) \quad\textnormal{and}\quad\rel(u) = \rel(a) \, .   \]
%


Further, call two sequences~$a$ and~$b$ {\em equivalent}, and write~$a \sim b$, when there is a bijection~$\phi \colon \Delta \to \Delta '$, where~$\Delta $ resp.\ $\Delta '$ is the alphabet of~$a$ resp.\ $b$, such that~$b_n = \phi(a_n)$ for all~$n \ge 1$. We call a sequence~$a$ {\em self-similar} if~$a \sim u$ for all~$u \in N(a)$. Self-similar sequences are homogeneous; to see this, note that when~$u \in N(a)$ one always has~$N(u) \subset N(a)$, and if~$a \sim u$ we draw~$|N(u)| = |N(a)|$ so that~$N(u) = N(a)$, and~$\rel(u) = \rel(a)$ is then obvious.

For example, if~$a$ is the Thue-Morse sequence, then~$N(a)$ has just two elements, namely~$a^{0, 0} = a$ and~$a^{1,1} = (a_{2n+1})_{n\ge 1}$. We have~$a^{1, 0} = (a_{2n})_{n\ge 1} = a$ so~$(1, 0) \in \rel(a)$, giving an example of relation; more importantly, we have~$a^{1,1} \sim a$ since~$a^{1, 1}$ is simply obtained by switching~$A$ and~$B$. So~$a$ is self-similar.

 Our first result is the following (results for general~$p$-sequences are presented below).

\begin{thm} \label{thm-intro-iff}
  Let~$(a)$ be a~$2$-automatic sequence. Then~$(a)$ is homogeneous if and only if it can be produced by an automaton whose underlying graph is the Cayley graph associated to a group~$G$, and such that the map~$\tau \colon G \to \Delta $ has the following property: the subgroup~$H$ of elements~$h\in G$ verifying~$\tau (hg) = \tau(g)$, for all~$g \in G$, is normal in~$G$.

Moreover, $(a)$ is self-similar if and only if it can be produced by an automaton whose underlying graph is the Cayley graph associated to a group~$G$, and such that the alphabet can be identified with~$G/K$ for some subgroup~$K$, with~$\tau \colon G \to G/K$ the natural map.
\end{thm}

(In the text there is also an even weaker type of sequence, corresponding to a weaker type of symmetry. The automata involved are those which have been called {\em permutation automata} or {\em reversible automata} in the literature.)

The second statement is probably the most satisfying: a kind of symmetry in the graph translates precisely into another sort of symmetry within the sequence. Note also that the condition on the subgroup~$H$, in the first statement, is satisfied if~$\tau $ is conjugation-invariant in the sense that~$\tau (y^{-1}x y) = \tau (x)$. For example $\tau $ may be the character of a representation. Finally, we point out that one may replace~$G$ by~$G/H$ (since~$H$ is normal), and obtain a smaller automaton still producing~$a$. We turn to questions of minimality.

\subsection{Canonical automata}

These results are entangled with another question. When~$a$ is an automatic sequence, is there a way to {\em canonically} construct an automaton which produces~$a$? It turns out that the answer is affirmative for homogeneous $2$-sequences (and so also self-similar sequences).

We shall in fact define an oriented graph~$\Gamma (a)$ from any~$p$-automatic sequence~$a$ (for any~$p$), whose set of vertices is~$N(a)$. We also define a monoid~$G(a)$ of self-maps of~$N(a)$, having generators~$t_0, \ldots, t_{p-1}$, and~$\Gamma (a)$ is the ``Schreier graph'' of this monoid (we define Schreier graphs below). There is a distinguished vertex in~$\Gamma (a)$, namely~$a$ itself. We shall prove that whenever~$G(a)$ is actually a group, and not just a monoid, then there is a natural map~$\tau \colon N(a) \to \Delta $ (indeed~$\tau (u)$ is the first term of the sequence~$t_1^{-1}(u)$). Thus~$\Gamma (a)$ is a full-blown automaton in this case.

We establish the following.

\begin{thm} \label{thm-intro-canonical}
Let~$(a)$ be a~$2$-automatic sequence. Then~$(a)$ is homogeneous if and only if~$G(a)$ is a group and~$\Gamma (a)$ is its Cayley graph. In this case~$\Gamma (a)$ produces~$(a)$. 
\end{thm}

For a general~$p$, the statements of the two theorems stated so far do not hold in such generality. Indeed consider the sequence~$a_n=$ the leftmost digit of~$n$ when written in base~$p$. One can show that~$a$ is~$p$-automatic, with~$N(a) = \{ a \}$, so that the graph~$\Gamma (a)$ has only one vertex. Clearly it cannot be turned into an automaton producing~$a$, unless~$p=2$.

However, a great deal remains true. Here, let us state the following.

\begin{thm}
Let~$G$ be a group with distinguished generators~$t_0, \ldots, t_{p-1}$. Suppose the corresponding Cayley graph is turned into an automaton, with initial state~$1$, by means of the map~$\tau \colon G \to G/K $ for some subgroup~$K$. Let~$(a)$ be the~$p$-automatic sequence produced.

Then~$(a)$ is self-similar, the monoid~$G(a)$ is a group and~$\Gamma (a)$ is its Cayley graph. If moreover the intersection of all the conjugates of~$K$ is trivial, then~$G$ can be identified with~$G(a)$, and~$\Gamma (a)$ can be identified with the graph underlying the automaton.
\end{thm}



A natural choice is~$K= \{ 1 \}$, and the alphabet is then~$G$ itself. Both examples at the beginning of this introduction are of this kind: we have noticed that the graphs on Figure~$1$ and Figure~$2$ are Cayley graphs, and since the labeling map~$\tau$ is injective, we may see it as simply giving names to the elements of the group (for example on Figure~$2$ the unit of~$G$ is called~$d_1$, the element~$t_1= t_5$ is called~$d_5$, the element~$t_2= t_3= t_4$ is called~$d_3$).

\begin{rmk}
If we combine the results stated so far, for~$p=2$, we see that a~$2$-automatic, self-similar sequence determines, and is entirely determined up to equivalence by, a finite group~$G$, two generators~$t_0$ and~$t_1$, and a conjugacy class of subgroups such that the intersection of all the subgroups in the class is trivial. This is tantamount to specifying a conjugacy class of subgroups of finite index in the free groups on two generators~$\langle t_0, t_1 \rangle$. In turn, many interesting objects, comprising the theory of {\em dessins d'enfants} as in~\cite{pedro}, are in bijection with~$2$-automatic, self-similar sequences. 
\end{rmk}

\subsection{Rational fractions}

We give an application to the study of certain rational fractions associated to automatic sequences. 

The following holds in full generality. It may be known to the experts, but the author was not able to find a statement in the literature.

\begin{thm}
Let~$a$ be a~$p$-automatic sequence, and assume that the elements of the alphabet~$\Delta $ are taken in a ring. Define 
\[ L(a, x) = \sum_{n \ge 1} a_n x^{\ell(n)} \, ,   \]
where~$\ell(n)$ is the length of~$n$ when written in base~$p$. Then~$L(a, x)$ is a rational fraction.

Moreover, there is an explicit formula for computing~$L(a, x)$, involving the incidence matrix of the graph~$\Gamma (a)$.
\end{thm}

The reason why this result has not received much attention is perhaps that the graph~$\Gamma (a)$, in general, is difficult to determine explicitly. This is where our previous considerations will be useful: we have just given a recipe for constructing an automatic sequence from a group~$G$ with distinguished generators, in such a way that~$\Gamma (a)$ is just the Cayley graph of~$G$ (and thus is explicitly known to us from the outset). We have also pointed out that our two running examples are of this kind. 

For the Thue-Morse sequence~$a$, seeing the alphabet~$\{ A, B \}$ as a subset of the ring~$\z[A, B]$, we find 
\[ L(a, x) = \frac{- A x^{2} + B x^{2} -  B x}{2 x - 1} \, .   \]

For the~$7$-automatic sequence~$b$ produced by the automaton on Figure~$2$, we work with the ring~$\z[d_1, d_2, \ldots, d_6]$ and find:
\[ L(b, x) = \frac{d_1 P_1 + d_2 P_2 + d_3 P_3 + d_4 P_4 + d_5P_5 + d_6 P_6} {-441 x^{6} - 336 x^{5} - 300 x^{4} + 128 x^{3} + 24 x^{2} - 12 x + 1}  \]
where 
\[ P_1=497 x^{6} + 380 x^{5} + 136 x^{4} - 80 x^{3} + 2 x^{2} + x \, ,   \]
\[ P_2=-112 x^{6} + 96 x^{5} + 70 x^{4} - 63 x^{3} + 9 x^{2} \, ,  \]
\[ P_3=28 x^{6} - 148 x^{5} + 102 x^{4} + 42 x^{3} - 27 x^{2} + 3 x \, ,   \]
\[ P_4= 98 x^{6} - 179 x^{5} + 105 x^{4} - 28 x^{3} + 4 x^{2} \, ,   \]
\[ P_5 = -203 x^{6} + 123 x^{5} + 68 x^{4} + 28 x^{3} - 18 x^{2} + 2 x \, ,   \]
and 
\[ P_6 = 70 x^{6} + 70 x^{5} - 175 x^{4} + 35 x^{3} \, .   \]

These rational fractions contain information, in particular, about the frequencies of letters. Write~$a[n, d_i]$ for the number of occurences of the symbol~$d_i$ among the first~$n$ terms of the sequence~$a$. The {\em frequency} of~$d_i$ is 
\[ \lim_{n \to \infty} \frac{a[n, d_i]} {n} \, ,   \]
when this limit exists. We will be able to prove or disprove the existence of 
\[ \lim_{n \to \infty} \frac{a[p^n, d_i]} {p^n} \, ,  \]
as well as evaluate its value, when~$a$ is~$p$-automatic and we know~$L(a, x)$. Let us be content, in this introduction, with the following.

\begin{thm}
Suppose~$a$ is~$p$-automatic, on the alphabet~$\{ d_1, \ldots, d_k \}$, and let
\[ L(a, x) = \frac{\sum_i d_i P_i(x)} {D(x)}  \]
where~$P_i(x), D \in \z[x]$. Assume that the roots of~$D$ have absolute value~$\ge \frac{1} {p}$, that the only root of~$D$ of absolute value~$\frac{1} {p}$ is~$\frac{1} {p}$ itself, and that this root is simple. Then 
\[ \lim_{n \to \infty} \frac{a[p^n, d_i]} {p^n} = \frac{P_i(\frac{1} {p})} {\sum_j P_j(\frac{1} {p})} \, .   \]
\end{thm}

For our sequence~$b$ related to the Apéry numbers, the denominator is 
\[ \left(-1\right) \cdot (3 x + 1) \cdot (7 x - 1) \cdot (x^{2} + x + 1) \cdot (21 x^{2} -9 x + 1) \, ,   \]
and the theorem applies. The polynomials~$P_i$ might look very different from one another, but for all~$i$ we have 
\[ P_i\left( \frac{1} {7} \right) = \frac{570} {16807} \, .   \]
Thus the ``frequencies'' all agree, and since they sum up to~$1$, they must be equal to~$\frac{1} {6}$.

It is tempting to presume that, a self-similar sequence being so symmetric, the frequencies will always be~$\frac{1} {n}$ where~$n$ is the order of the group ($=$ the size of the alphabet, the number of vertices). However we have a counter-example, of a~$2$-automatic sequence, which is self-similar with~$G(a)$ of order~$8$, for which the denominator has both~$\frac{1} {2}$ and~$-\frac{1} {2}$ as roots. Our more general statement implies that the limit (as above) does not exist, the values oscillating between~$\frac{1} {6}$ and~$\frac{1} {12}$. Yet the average is still $\frac{1} {2}( \frac{1} {6} + \frac{1} {12}) = \frac{1} {8}$.

\subsection{Organization}

We complete the definitions in Section 2. The theorems stated in the introduction are proved in Section 3, albeit in a different order. Section 4 presents the rational fractions associated to automatic sequences, and their application to the computation of frequencies.

\section{A few more preliminaries}

The rather long introduction contained a number of definitions which will not be repeated. In this section we complete the set of definitions and make a few technical points.

\subsection{The need for sequences starting from~$1$}

Consider the following~$2$-automaton:

\begin{center}
\includegraphics[width=.75\textwidth]{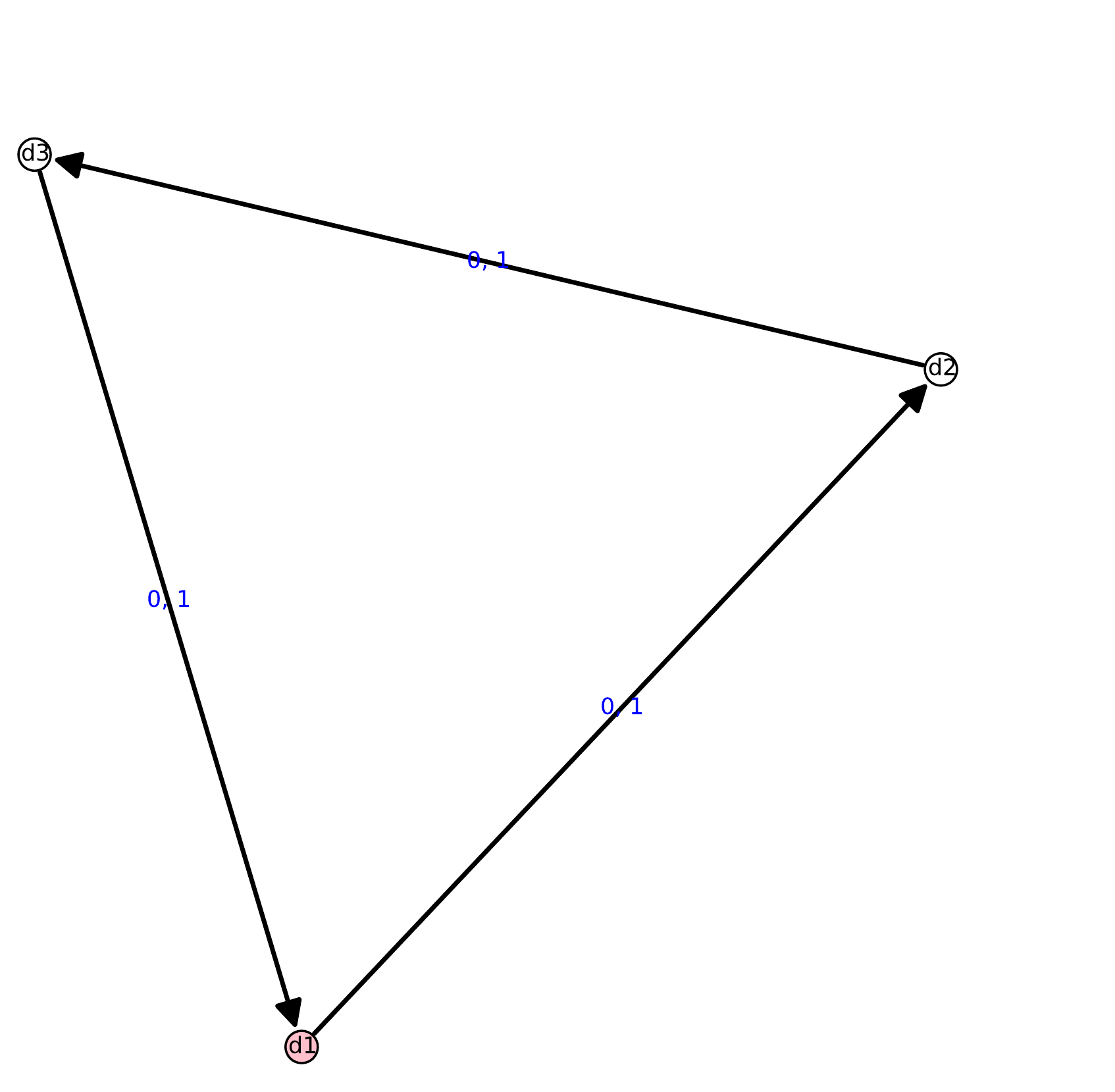}
\end{center}

The usual practice is to associate to this automaton a sequence~$(a_n)_{n \ge 0}$ following the recipe we gave in the introduction for~$n \ge 1$, and with~$a_0=$ the initial state. We leave it as an exercise to the reader to check that~$a$ has then~$9$ subsequences of the form~$(a_{2^in + j})_{n \ge 0}$ with~$j < 2^i$ ; moreover, these~$9$ sequences come in~$3$ groups, each consisting of three sequences which agree for~$n \ge 1$ but have different~$0$-th terms.

By contrast, if we consider~$(a_n)_{n \ge 1}$ and the subsequences belonging to~$N(a)$ as above, then there are just three of them. In general one has:

\begin{lem} \label{lem-Q-to-Na}
Let~$Q$ be the set of states of a $p$-automaton producing the sequence~$(a_n)_{n \ge 1}$. The map 
\[ Q \longrightarrow N(a) \, ,  \]
which associates to~$q \in Q$ the sequence obtained by taking~$q$ to be the initial state, is surjective.
\end{lem}

\begin{proof}
Let~$q_0$ be the original initial state (producing~$a$). To produce the sequence~$(a_{p^in + j})_{n \ge 1}$, follow the arrows from~$q_0$ according to the digits of~$j$ written in base~$p$, padded with~$0$'s on the left so that~$i$ moves are made.
\end{proof}

The example above shows that the lemma would fail for sequences starting from~$0$, since~$Q$ has~$3$ elements and we have found~$9$ subsequences. 

On the other hand, for sequences starting from~$1$, the lemma proves that~$N(a)$ is finite when~$a$ is automatic. The converse is also true: when~$N(a)$ is finite, then~$a$ is~$p$-automatic (note that the number~$p$ is used in the definition of~$N(a)$ even if it is absent from the notation). The classical proof (see for example~\cite{allouche}) is usually given for sequences starting from~$0$, but it is trivial to deduce the same statement for sequences starting from~$1$.

From now on, all sequences will start from~$1$. We may use the simple phrase ``subsequence of~$a$'' to mean specifically a subsequence of the form~$(a_{p^i n + j})_{n \ge 1}$ with~$j < p^i$, when the context makes it clear.

\subsection{Schreier graphs}
We shall introduce sequences which are more general than homogeneous sequences (namely sequences ``with global relations of all types'', see below), corresponding to a class of graphs more general than Cayley graphs.

Let~$G$ be a group (typically finite) with distinguished generators~$t_0, t_1, \ldots, t_{p-1}$. If~$G$ acts on a finite set~$Q$, then one may form the {\em Schreier graph}, whose set of vertices is~$Q$, and which includes a directed arrow bearing the label~$i$ between each~$q \in Q$ and its image under the action of~$t_i$, for~$0 \le i < p$. For example, when~$Q= G$ and the action is simply multiplication (on the right), the corresponding Schreier graph is just the Cayley graph. On the other hand, if~$G$ is the dihedral group of orthogonal, planar transformations preserving a square, generated by a rotation~$t_0$ and a symmetry~$t_1$, then the Schreier graph corresponding to the action on the four corners of the square is

\begin{center}
\includegraphics[width=.75\textwidth]{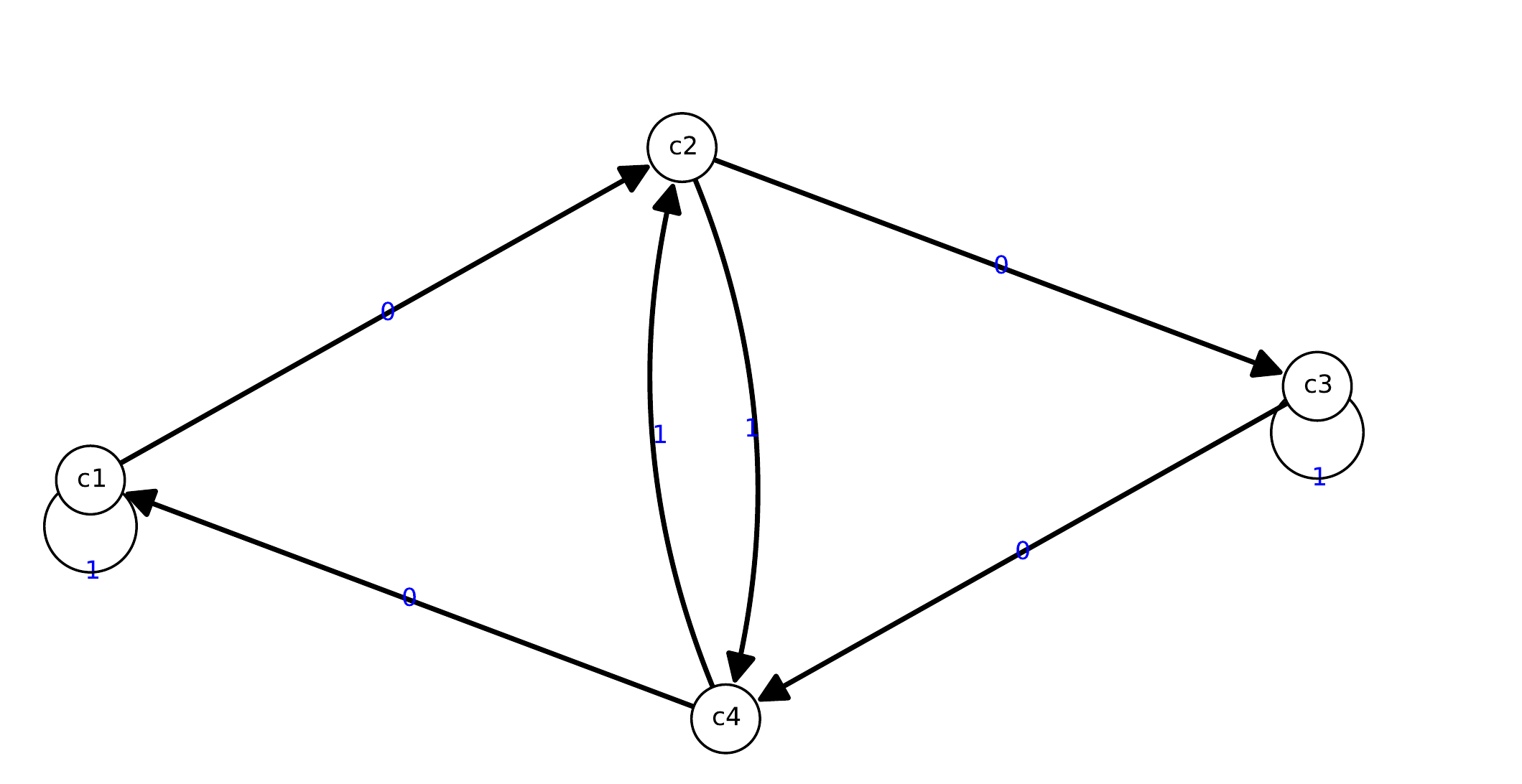}
\end{center}

It is easy to decide whether a given directed graph is a Schreier graph. This happens if and only if the property observed above for Cayley graph holds, namely, at each vertex there is exactly one arrow with the label~$i$ going out, and also exactly one arrow with the label~$i$ coming in. One can then construct a permutation~$t_i$ of~$Q$ whose action is dictated by the arrows in the graph, and define~$G$ to be the group generated by the~$t_i$'s.

Let us say that a directed graph is {\em connected} if there is at least one vertex~$q_0$ such that any other vertex~$q$ can be reached from~$q_0$ by following directed arrows. The next lemma is well-known, and pretty trivial, but we include it for convenience.

\begin{lem} \label{lem-schreier-is-cayley}
Let~$G$ be a finite group, with distinguished generators, acting on the set~$Q$. Then the corresponding Schreier graph can be identified with the Cayley graph (non-canonically) if and only if it is connected and~$|Q| = |G|$.
\end{lem}

\begin{proof}
Let~$q_0$ be as in the definition of connectedness (the arbitrary choice is why the identification will not be canonical). The map~$G \to Q$ which maps~$g$ to~$q_0^g$ (the result of letting~$g$ act on~$q_0$) is surjective by connectedness, and so also injective for reasons of cardinality. The desired identification follows.

The converse holds trivially. Moreover we see {\em a posteriori} that any~$q_0 \in Q$ could have been taken.
\end{proof}

Finally, note that the above definitions make sense if~$G$ is only a monoid rather than a group. The graph~$\Gamma (a)$, alluded to in the introduction, is precisely the Schreier graph of a certain monoid. We turn to this.

\subsection{The graph~$\Gamma  (a)$}

Let~$a$ be~$p$-automatic. We define maps 
\[ t_i \colon N(a) \longrightarrow N(a)  \]
by 
\[ t_i(u) = (u_{p n + i})_{n \ge 1}  \]
for~$0 \le i < p$. (The notation~$t_i$ is for ``times~$p$ plus~$i$''.) The monoid generated by these (a submonoid of the monoid of all self-maps of~$N(a)$) will be written~$G(a)$. We are particularly interested in situations when~$G(a)$ is a group, or equivalently when each~$t_i$ is a bijection, hence the notation. Typical elements of~$G(a)$ will be written~$g$ or~$h$.

Many formulae will be simplified by the following convention. We define the composition of~$G(a)$ as follows: $gh$ means first~$g$ and then~$h$. Accordingly we will write~$u^g$ rather than~$g(u)$ (for~$g \in G(a)$, $u \in N(a)$), and we have~$u^{gh} = (u^g)^h$.

Place an arrow with the label~$i$ between~$u$ and~$u^{t_i}$, for all~$u$ and all~$i$. We now have the directed graph~$\Gamma (a)$, with labeled arrows, and also with the distinguished vertex~$a$. Note that~$\Gamma (a)$ is the Schreier graph for the action of the monoid~$G(a)$ on the set~$N(a)$.

It is important to realize that~$\Gamma (a)$ is connected. More precisely, there is a sequence of directed arrows leading from~$a$ to any~$u \in N(a)$. This will follow from the computations which we describe now.

\subsection{Basic calculations} \label{subsec-basic-calculations}

The following calculations will be used many times, often implicitly. Let~$p$ be fixed, as ever. We define operations~$t_i$ for~$0 \le i < p$ on polynomials by 
\[ f^{t_i} = f(pX + i) \qquad\qquad ~\textnormal{for}~ f \in \z[X] \, .  \]
When~$w$ is a word in the alphabet~$\{ t_0, \ldots, t_i \}$, then~$f^{w}$ has the obvious meaning. We are solely interested in the polynomials~$X^w$, which are all of the form $p^i X + j$ with~$j < p^i$, by an immediate induction; we claim that each such polynomial is obtained for some~$w$. This will prove that~$\Gamma (a)$ is connected.

In fact, we shall be more precise. To the polynomial~$p^i X + j$ with~$j < p^i$ we associate a word in the alphabet~$\{ 0, 1, \ldots, p-1 \}$, denoted~$[p^iX + j]$, by the following rule: write~$j$ in base~$p$, and pad the results with~$0$'s on the left so as to have~$i$ digits in total. For example with~$p=2$ one has~$[X]=$ the empty word, $[2X]= 0$, $[2X+1] =1$, $[16X+3] = 0011$.

An obvious remark is that~$[p^iX + j] = [p^k X + \ell]$ imply that~$i=k$ and~$j = \ell$. Besides, one has 
\[ [(p^i X + j)^{t_s}] =s [p^i X + j] = ~\textnormal{the word}~[p^iX + j] ~\textnormal{with an}~s ~\textnormal{on the left} \, .   \]
As a result, if~$[p^i X + j] = d_{i-1} \cdots d_1 d_0$, then by putting~$w = t_{d_0} t_{d_1} \cdots t_{d_{i-1}}$, we have~$X^w = p^i X + j$. (Incidentally this show that~$w$ can be recovered from~$X^w$, and so~$X^w = X^{w'}$ imply~$w = w'$. We will not have much use for this remark.)

What is important is that this establishes the claim. The relationship with the connectedness of~$N(a)$ is clear, since we have given a definition of~$u^w$ for~$u \in N(a)$, and we have~$u^w = u ^{(i, j)}$ where~$X^w = p^i X + j$ (the notation~$u ^{(i,j)}$ is explained in the introduction). We will sometimes go back and forth between the~$w$- and~$(i,j)$-notation.

The operation~$t_1$ is (slightly) more important than the others (since the theory is richer for~$2$-automatic sequences), so one final comment will be handy. When we put~$X=1$, the number~$p^i + j$ is written in base~$p$ as~$[(p^i X + 1)^{t_1}] = 1 [p^i X + j]=$ the word~$[p^iX + j]$ with a~$1$ on the left.

\subsection{The type of a relation; global relations}

The {\em type} of a relation~$(i,j) \in \rel(a)$ is the leftmost digit in~$[p^iX + j]$, or equivalently, is~$0$ when~$j < p^{i-1}$ and is the leftmost digit of~$j$ written in base~$p$ otherwise. Let us say that~$(i,j)$ is a {\em global relation} for~$a$ if 
\[ (i,j) \in \bigcap_{u \in N(a)} \rel(u) \, ;  \]
and let us say that~$a$ {\em has global relations of all types} if there is (at least) a global relation~$(i,j)$ of type~$r$ for each~$0 \le r < p$.

\begin{lem} \label{lem-homogeneous-imples-global}
Let~$(a)$ be a~$p$-automatic sequence. \begin{enumerate}
\item Let~$r$ be an interger with~$0 \le r < p$. Then there is~$u \in N(a)$ such that~$\rel(u)$ contains a relation of type~$r$.
\item If~$(a)$ is homogeneous, then~$(a)$ has global relations of all types.
\end{enumerate}
\end{lem}

Property (2) of this lemma will imply that, in practice, we will almost never have to worry about the types of relations.

\begin{proof}
(1) The elements~$a$, $a^{t_r}$, $(a^{t_r})^{t_r} = a^{t_r^2}$, $\ldots $, $a^{t_r^k}$, $\ldots$, defined for all~$k \ge 0$, cannot be all different, since they are taken from the finite set~$N(a)$. So there must exist~$k \ge 0$ and~$m > 0$ such that 
\[ a^{t_r^{k+m}} = a^{t_r^k} \, .   \]
In other words~$u^{t_r^m} = u$ where~$u = a^{t_r^k}$. Let~$(i,j)$ be such that~$X^{t_r^m} = p^iX + j$, as above ; since~$[p^i X + j] = rrrr \cdots  $ ($m$ times), the type of~$(i,j)$ is certainly~$r$. Moreover~$u ^{(i,j)} = u^{t_r^m} = u$, so~$(i,j) \in \rel(u)$.

(2) When~$a$ is homogeneous, $\rel(u)$ is the same for all~$u \in N(a)$. By the first point, this set contains relations of all types.
\end{proof}

\section{The main theorems}

We now have all the tools to embark on a proof of the theorems stated in the introduction. These will be obtained in a completely different order.

\subsection{Sequences with global relations of all types}

These sequences seem perhaps less interesting for their own sake than homogeneous or self-similar sequences (and condition (R1) below is admittedly a bit artificial). However it is technically quite easy to start with their properties, and subsequently refine the results to deal with other types of sequences.

\begin{prop} \label{prop-global-relations-all-types}
Let~$(a)$ be a~$2$-automatic sequence. Then the following statements are equivalent.

\begin{enumerate}
\item[(R1)] The sequence $(a)$ has global relations of all types.

\item[(R2)] The monoid $G(a)$ is a group.

\item[(R3)] The sequence $(a)$ can be produced by an automaton whose underlying graph is the Schreier graph of a group~$G$ with distinguished generators, acting faithfully and transitively on a set~$Q$.
\end{enumerate}

Moreover, we have the following minimality statement. Suppose that~$G$, $Q$ and~$\tau $ (the labeling map $G \to \Delta $) are as in (R3). Then~$|Q| \ge |N(a)|$, with equality precisely when the condition below holds:
\[ ~\textnormal{if}~ \tau (q_0^{hg}) = \tau (q_0^g) ~\textnormal{for all}~g \in G ~\textnormal{then}~ q_0^h = q_0 \, . \tag{$\dagger$} \]
(Here~$q_0$ is the initial state, and $q_0^g$ is the image of~$q_0$ under the action of~$g \in G$.)

In this case $G$ can be identified with~$G(a)$, and~$Q$ with~$N(a)$. This identification preserves the distinguished generators. 
\end{prop}

Recall that an action of~$G$ is called {\em faithful} if no element of~$G$ except for the identity acts as the identity permutation. The study of any group action reduces to the study of a faithful action.

\begin{proof}
Suppose (R1) holds. In graph-theoretic words, condition (R1) says that each vertex of~$N(a)$ is at the end of an arrow bearing the label~$r$, for each~$r$. Thus~$t_r$ is a surjective map~$N(a) \to N(a)$, and it must be a bijection since~$N(a)$ is finite, $a$ being automatic. So~$G(a)$ is a group, which is (R2).

Now assume (R2), and let us prove that (R3) holds with~$G = G(a)$ and~$Q= N(a)$. The Schreier graph~$\Gamma (a)$ can be turned into an automaton if we  define $ \tau \colon N(a) \to \Delta $ by~$\tau(u)=$ the first term of~$u^{t_1^{-1}}$. (Indeed the sequence~$u$ is of the form~$u_n = v_{2n + 1}$ for a unique~$v \in N(a)$, or in other words there is an arrow with the label~$1$ from~$v$ to~$u$, so surely we must define the label at~$u$ to be~$v_1$.) We must prove that~$\Gamma (a)$ then produces~$a$.

Indeed, suppose~$n= d_id_{i-1} \cdots d_1d_0$ in base~$2$, let~$w = t_{d_0} \cdots t_{d_{i}}$. Following the arrows accordingly in~$\Gamma (a)$ starting from~$a$ leads us to~$a^w$. Now {\em we know that~$d_i = 1$}, since~$p=2$, so let~$w= w' t_1$ and let~$v= a^{w'}$, so that~$a^w = v^{t_1}$. The label~$\tau(a^w) $ is~$v_1$ by definition, and we must show that it is~$a_n$. However~$v= a^{w'} = a^{(i,j)}$ where~$X^{w'} = 2^i X + j$, with notation as in~\S\ref{subsec-basic-calculations} ; that is~$v_n = a_{2^in + j}$ for~$n \ge 1$, so in particular~$v_1= a_{2^i + j}$. Here the very last remark in~\S\ref{subsec-basic-calculations} is that~$2^i + j$, when written in base~$2$, is~$1[2^i X + j] = 1d_{i-1} d_{i-2} \cdots d_1 d_0$, which is also how~$n$ is written, so~$2^i + j= n$. This proves (R3). 

We check the minimality condition. That~$\tau (a^{hg}) = \tau (a^g)$ for all~$g \in G$ is equivalent to claiming that the sequence obtained from the automaton with~$a^h$ as the initial state is again~$a$. However, the map from Lemma~\ref{lem-Q-to-Na} is now a surjective map~$N(a) \to N(a)$, which must also be injective, showing that~$a^h= a$. 

That (R3) implies (R1) is almost obvious. The sequence~$a$ is produced, when R3 holds, by an automaton in which each state is at the end of an arrow marked~$r$, for all~$r$. Using the surjective map~$Q \to N(a)$ from Lemma~\ref{lem-Q-to-Na}, we deduce that any~$u \in N(a)$ is of the form~$v^{t_r}$, for all~$r$, and that is condition (R1) reworded.

Here we use the fact that the map~$Q \to N(a)$ is compatible with the generators~$t_i$, in the sense that if it maps~$q$ to~$u$, then it maps~$q^{t_i}$ to~$u^{t_i}$ ; this holds obviously, even when there is no group in sight and~$q^{t_i}$ is simply taken to mean the state at the end of the arrow with label~$i$ originating at~$q$.


The same observation will prove the last statements. For suppose condition~$(\dagger)$ is satisfied. We claim that the map~$Q \to N(a)$ is then injective. Since the action of~$G$ is transitive, any~$q \in Q$ is of the form~$q_0^h$, where~$q_0$ is the initial state. Suppose~$h, k \in G$ are such that the sequences produced by choosing~$q_0^h$ and~$q_0^k$ as initial state, respectively, coincide. Then~$\tau(q_0^{hx}) = \tau(q_0^{kx})$ for all~$x \in G$, and in particular for~$x= k^{-1} g$ we find~$\tau(q_0^{hk^{-1} g}) = \tau(q_0^g)$ for all~$g\in G$, so~$q_0^{hk^{-1}} = q_0$ by~$(\dagger)$, and~$q_0^h = q_0^k$. This proves the claim.

This provides the required identification of~$Q$ with~$N(a)$, compatible with the bijections denoted by~$t_i$ on each of these sets. The corresponding permutation groups can also be identified, and these are~$G$ (because the action is assumed to be faithful) and~$G(a)$ (by definition).

It remains to note that when~$(\dagger)$ does not hold, the map~$Q \to N(a)$ is not injective (but still surjective), so~$Q$ has more elements than~$N(a)$. 
\end{proof}

Large parts of this proof work in the general case of~$p$-automatic sequences for any~$p$. We state this separately:

\begin{prop}
Let~$G$ be a finite group generated by~$t_0, \ldots, t_{p-1}$, and suppose that~$G$ acts faithfully and transitively on a finite set~$Q$. Suppose that the corresponding Schreier graph is turned into an automaton by choosing an initial state~$q_0 \in Q$ and a labeling map~$\tau \colon Q \to \Delta $ satisfying the same minimality condition $(\dagger)$ as above. Let~$(a)$ be the~$p$-automatic sequence produced.

Then~$(a)$ has global relations of all types, the monoid~$G(a)$ is a group which can be identified with~$G$, and~$\Gamma (a)$ can be identified with the graph underlying the automaton.
\end{prop}

\begin{ex}
We prove that Proposition~\ref{prop-global-relations-all-types} does not hold for~$p > 2$. Indeed for any~$p$, let~$a_n=$ the last digit of~$n$ written in base~$p$. All the subsequences of the form~$(a_{p^i n + j})_{n \ge 1}$ coincide with the original sequence, so~$N(a)= \{ a \}$ and~$a$ is~$p$-automatic. However the graph~$\Gamma (a)$ has only one vertex, and no matter the decoration we elect to place on~$\Gamma (a)$, the automaton obtained can only produce the constant sequence. This will be different from~$a$ except for~$p=2$. So~$a$ cannot be produced by~$\Gamma (a)$, even though it certainly has all the relations one could ask for (so R1 holds).
\end{ex}

\subsection{Homogeneous sequences}

\begin{thm} \label{thm-homogeneous}
Let~$(a)$ be a~$2$-automatic sequence. Then the following statements are equivalent.

\begin{enumerate}
\item[(H1)] The sequence $(a)$ is homogeneous.

\item[(H2)] The monoid $G(a)$ is a group, and~$\Gamma (a)$ is its Cayley graph.

\item[(H3)] The sequence $(a)$ can be produced by an automaton whose underlying graph is the Cayley graph of a group~$G$ with distinguished generators (with~$1$ as the initial state), and with the property that the subgroup~$H$ of elements satisfying~$\tau (hg) = \tau (g)$ for all~$g \in G$ is normal in~$G$.
\end{enumerate}

Moreover in (H3) we can arrange to have~$G= G(a)$, and~$H= \{ 1 \}$. Conversely, suppose we start with~$G$ as in (H3) and that~$H= \{ 1 \}$. Then $G$ can be identified with~$G(a)$. This identification preserves the distinguished generators.
\end{thm}

\begin{proof}
We start by noting that (H1) implies (R1) (Lemma~\ref{lem-homogeneous-imples-global}), (H2) implies (R2) and (H3) implies (R3) (trivially). So by Proposition~\ref{prop-global-relations-all-types}, we may conduct the proof assuming that all three properties (R1-R2-R3) hold.

Assume (H1). For~$u \in N(a)$, we let~$\stab (u)$ denote its stabilizer under the action of~$G(a)$ (that is, the group of~$g \in G(a)$ such that~$u^g = u$). The fact that~$\Gamma (a)$ is connected (or equivalently, that the action of~$G(a)$ is transitive) implies that the various groups~$\stab (u)$ are all conjugate as~$u$ runs through~$N(a)$. Also, the intersection of all these is the trivial subgroup~$\{ 1 \}$, since it consists of elements~$g$ fixing everything in~$N(a)$, while~$G(a)$ is a group of permutations of this set by definition.

However (H1) says that~$\rel(u)$ does not depend on the choice of~$u$, and if follows that~$\stab (u)$ is also independent of the particular subsequence~$u$. Finally we see that the subgroups~$\stab(u)$, being all equal with trivial intersection, are all trivial. For cardinality reasons, this implies that~$|N(a)| = |G(a)|$, and so~$\Gamma (a)$ can be identified with the Cayley graph of~$G(a)$ (Lemma~\ref{lem-schreier-is-cayley}). Thus (H2) holds.

When (H2) is assumed, Proposition~\ref{prop-global-relations-all-types} gives (H3) with~$G= G(a)$ and~$H= \{ 1 \}$ (so that, in particular, $H$ is normal in $G$). The ``moreover'' statements will also be clear, but we should finish the equivalence first.

So assume (H3). Replacing~$G$ by~$G/H$ if necessary (which makes sense since~$H$ is assumed to be normal), we are reduced to the case~$H= \{ 1 \}$. By Proposition~\ref{prop-global-relations-all-types}, we can and we do identify~$G$ with~$G(a)$ and~$Q$ with~$N(a)$. Following arrows leading from any~$u \in N(a)$ to~$a$ (and in a Cayley graph, this is always possible), we see that~$a \in N(u)$, so~$N(a) = N(u)$. Further, consider the condition~$(i,j) \in \rel (u)$. It is equivalent to the requirement that the element of~$G(a)$ obtained as the word in the generators~$t_r$ corresponding to the word~$[2^i X + j]$ as in~\S\ref{subsec-basic-calculations} be an element of~$\stab(u)$ (this is just a game with notation). However for all~$u$ we have~$\stab (u) = \{ 1 \}$ since we are in a Cayley graph, and we see that the condition~$(i,j) \in \rel(u)$ actually does not depend on~$u$. So~$\rel(u) = \rel(a)$. We have (H1).  
\end{proof}

For general~$p$-sequences, what remains true is:

\begin{thm}
Let~$G$ be a group with distinguished generators~$t_0, \ldots, t_{p-1}$. Suppose the corresponding Cayley graph is turned into an automaton, with initial state~$1$, by means of a map~$\tau \colon G \to \Delta $ such that the subgroup~$H$ of elements satisfying~$\tau (hg) = \tau (g)$ for all~$g \in G$ is normal in~$G$. Let~$(a)$ be the~$p$-automatic sequence produced.

Then~$a$ is homogeneous, the monoid~$G(a)$ is a group and~$\Gamma (a)$ is its Cayley graph. If moreover~$H = \{ 1 \}$, then~$G$ can be identified with~$G(a)$, and~$\Gamma (a)$ can be identified with the graph underlying the automaton.
\end{thm}

\subsection{Self-similar sequences}

\begin{thm}
Let~$a$ be a~$2$-automatic sequence, on an alphabet~$\Delta $ (we assume that all the letters of~$\Delta $ are actually used). The following statements are equivalent.

\begin{itemize}
\item[(S1)] The sequence $(a)$ is self-similar.
\item[(S2)] The sequence~$(a)$ can be obtained by an automaton whose underlying graph is the Cayley graph of a group~$G$ with distinguished generators (with~$1$ as the initial state), and whose labeling map is the natural map~$\tau \colon G \to G/K \cong \Delta $ for some subgroup~$K$. 
\end{itemize}

Moreover, if the intersection of all the conjugates of~$K$ is trivial, then~$G$ can be identified with~$G(a)$.
\end{thm}

We stress that~$G/K$ is not~$K \backslash G$: it is the set of classes~$gK$ for~$g \in G$, and there is an action {\em on the left} on this set. Other actions so far have been {\em on the right}. This seems inevitable.

\begin{proof}
Assume (S1). The sequence~$(a)$, being self-similar, is also homogeneous. From Theorem~\ref{thm-homogeneous}, it can be produced by the automaton with underlying graph~$\Gamma (a)$, the Cayley graph of the group~$G=G(a)$, and we view~$\tau $ as defined on~$G$. Now any~$u \in N(a)$ is of the form~$a^h$ for a {\em unique} $h \in G$. The sequence~$a$ is given by~$a_n = \tau (g_n)$ where~$g_n \in G$ is written as the word in the~$t_r$'s corresponding to the digits of~$n$ is base~$2$ ; the sequence~$a^h$ is given by~$a^h_n = \tau (hg_n)$. 

By hypothesis, to each~$h$ we can attach a bijection~$\phi_h \colon \Delta \to \Delta $ such that~$\phi_h(a_n) = a^h_n$ for all~$n \ge 1$, or $\tau (hg_n) = \phi_h(\tau (g_n))$. Let us write~$h \cdot \delta  $ instead of~$\phi_h(\delta  )$, for~$\delta \in \Delta $, so that 
\[ \tau (hg) = h \cdot \tau (g) \]
for all~$g \in G$. The map~$\tau $ was assumed to be surjective, so this last relation ensures that~$h, \delta \to h \cdot \sigma $ is a (left) action of~$G$ on~$\Delta $ ; moreover this action is compatible with~$\tau $ and the left action of~$G$ on itself. The latter action is transitive, and thus, so must be the former. If~$K$ is the stabilizer of any point in~$\Delta $, we can then identify~$\Delta $ with~$G/K$. We have proved (S2).

The proof that (S2) implies (S1) is trivial: we {\em define} $\phi_h$ to be the action of~$h$ on~$\Delta  = G/K$.

The ``moreover'' statement follows from that of Theorem~\ref{thm-homogeneous}. 
\end{proof}

\begin{thm} \label{thm-self-similar-all-p}
Let~$G$ be a group with distinguished generators~$t_0, \ldots, t_{p-1}$. Suppose the corresponding Cayley graph is turned into an automaton, with initial state~$1$, by means of the map~$\tau \colon G \to G/K $ for some subgroup~$K$. Let~$(a)$ be the~$p$-automatic sequence produced.

Then~$(a)$ is self-similar, the monoid~$G(a)$ is a group and~$\Gamma (a)$ is its Cayley graph. If moreover the intersection of all the conjugates of~$K$ is trivial, then~$G$ can be identified with~$G(a)$, and~$\Gamma (a)$ can be identified with the graph underlying the automaton.
\end{thm}

\section{Rational fractions associated to automatic sequences}

\subsection{Defining the rational fractions}

Suppose~$(a_n)_{n \ge 1}$ is a~$p$-automatic sequence on an alphabet which is a subset of a ring~$R$. We define 
\[ L(a, x_0, \ldots, x_{p-1}) = \sum_{n \ge 1} \, a_n x_0^{\ell_0(n)} \cdots x_{p-1}^{\ell_{p-1}(n)} \in R[[x_0, \ldots, x_{p-1}]]  \]
where~$\ell_i(n)$ is the number of occurences of the digit~$i$ when writing~$n$ in base~$p$. We also define 
\[ L(a, x) = L(a, x, x, \ldots, x) = \sum_{n \ge 1} \, a_n x^{\ell(n)}  \in R[[x]]  \]
where~$\ell(n)$ is the length of~$n$ when written in base~$p$.

\begin{thm}
Under the assumption that~$a$ is~$p$-automatic, the power series~$L(a, x_0, \ldots, x_{p-1})$ is a rational fraction. In fact one can write~$L(a, x_0, \ldots, x_{p-1})$ as a coefficient of the column vector 
\[   (I - M)^{-1} C \]
where~$M \in \z[x_0, \ldots, x_{p-1}]$ is the weighted adjacency matrix of the graph~$\Gamma (a)$, we write~$I$ for the identity matrix, and~$C$ is a column vector whose entries are homogeneous polynomials in~$R[x_1, \ldots, x_{p-1}]$ of degree~$1$ (that is, linear forms).
\end{thm}

During the course of the proof we shall make precise the term ``weighted adjacency matrix'', as well as give an expression for~$C$.

\begin{proof}
Let~$u \in N(a)$. We write~$L = L(u, x_0, \ldots, x_{p-1})$ and~$L^{t_i} = L(u^{t_i}, x_0, \ldots, x_{p-1})$. Note that 
\[ \sum_{n \ge 1} u_{pn  + i} x_0^{\ell_0(pn+i)} \cdots x_{p-1}^{\ell_{p-1}(pn+i)} = x_i L^{t_i} \, .   \]
As a result, by partitioning the integers according to their values mod~$p$, one has 
\[ L = u_1 x_1 + \cdots + u_{p-1} x_{p-1} + x_1 L^{t_1} + \cdots + x_{p-1} L^{t_{p-1}} \, .   \tag{*} \]

Now fix an arbitrary order on~$N(a)$. We shall work with matrices and vectors which are indexed by the elements of~$N(a)$, with this order.

We start with the matrix~$M$ defined by 
\[ M_{u,v} = \sum_{u^{t_i} =v} x_i \, ,   \]
which we call ``the weighted adjacency matrix of~$\Gamma (a)$'' (see also the definition of the matrix~$A$ below). Next, define the column vectors~$\Lambda $ and~$C$ by
\[ \Lambda _u = L(u, x_0, \ldots , x_{p-1}) \, , \quad C_u = u_1x_1 + \cdots + u_{p-1} x_{p-1} \, .   \]
Equation (*) above can now be written 
\[ (I - M) \Lambda  = C \, .   \]
It remains to prove that~$I - M$, a matrix with entries in~$\z[x_0, \ldots, x_{p-1}]$, is invertible in the field of fractions~$\q[x_0, \ldots, x_{p-1}]$. For this it suffices to check that its determinant is not the zero polynomial, and in turn, it suffices to show this after evaluating at~$x_0 = x_1 = \cdots = x_{p-1} = x$.

In this situation~$M = xA$ where~$A_{u,v}$ is the number of indices~$i$ such that~$u^{t_i} = v$, and we call~$A$ the adjacency matrix of~$\Gamma (a)$. Now 
\[ \det(I - x A) = x^N \det(\frac{1} {x} I - A) = x^N \chi_A( \frac{1} {x}) \, .   \]
Here~$N = |N(a)|$ and~$\chi_A$ is the characteristic polynomial of~$A$, which has degree~$N$. The proof is complete.
\end{proof}

\begin{coro}
The power series~$L(a, x)$ is a rational fraction, given by a coefficient of the column vector
\[ (I - xA)^{-1} x T \, ,  \]
where~$A$ is the adjacency matrix of~$\Gamma (a)$ as defined above, and~$T_u = u_1 + \cdots + u_{p-1}$.
\end{coro}

\subsection{Examples}

We shall examine our two running examples. From now on when we deal with a sequence on an alphabet~$\{ d_1, \ldots, d_r \}$, we assume that the ring~$R$ is~$\z[d_1, \ldots, d_r]$.

\begin{ex}
The Thue-Morse sequence~$a$ is produced by the automaton on Figure 1. It is a Cayley graph, for the group of order two~$G= \{ \pm 1 \}$ with~$t_0 = 1$ and~$t_1 = -1$. The labeling map~$\tau $ is a bijection between the elements of~$G$ and~$\{ A, B \}$. Theorem~\ref{thm-self-similar-all-p} implies that~$a$ is self-similar and, most importantly here, $G(a)$ is nothing but~$G$ itself, and~$\Gamma (a)$ is the directed graph underlying the automaton. There are $2$ sequences in~$N(a)$, corresponding to the two possible initial states, and these are~$a$ and~$a^{t_1} = a ^{(1,1)} = (a_{2n+1})_{n \ge 1}$. Say~$N(a)$ is ordered so that~$a$ is the first element.

The matrix~$M$ is thus, using the letters~$x$ and~$y$ rather than~$x_0$ and~$x_1$: 
\[ M = \left(\begin{array}{rr}
x & y \\
y & x
\end{array}\right) \, .   \]
Writing~$L$ and~$L^{t_1}$ for the rational fractions associated to~$a$ and~$a^{t_1}$ respectively, we have 
\[ \Lambda =\left(\begin{array}{l}
L \\ L^{t_1}
\end{array}\right) \, , \qquad C = \left(\begin{array}{l}
x a_1 \\ x a^{t_1}_1
\end{array}\right) = \left(\begin{array}{c}
xB \\ xA
\end{array}\right) \, . \]

We only need to compute 
\[ (I-M)^{-1} = \frac{1} {-x^2 + y^2 + 2x - 1} \left(\begin{array}{rr}
x-1 & -y \\
-y & x-1
\end{array}\right)  \]
so that 
\[ \Lambda = (I-M)^{-1} C = \frac{1} {-x^2+y^2+2x -1} \left(\begin{array}{r}
- x y A + (x^{2} -  x) B \\
(x^{2} -  x) A  - x y B
\end{array}\right)
  \]
and in particular 
\[ L(a, x, y) = \frac{- xy A + (x^{2} -  x) B} {-x^2+y^2+2x -1}    \]
and 
\[ L(a, x) = L(a,x,x) =  \frac{- x^{2} A + (x^{2} -  x) B} {2x -1} \, .  \]

Notice how~$L(a^{t_1}, x)$ is obtained from~$L(a, x)$ by exchanging~$A$ and~$B$, and likewise for~$L(a^{t_1}, x, y)$ and~$L(a, x, y)$. 
\end{ex}

\begin{ex}
Let~$b$ be the~$7$-automatic sequence produced by the automaton on Figure 2. We compte~$L(b,x)$ (from now on we shall be interested in single-variable rational fractions, for simplicity). 

We have already observed in the introduction that the underlying graph is the Cayley group for a group~$G$, which is cyclic of order~$6$, and we have specified the generators~$t_0, \ldots, t_6$ as permutations. The 6 labels are all distinct, so we have a self-similar sequence, as dealt with by Theorem~\ref{thm-self-similar-all-p}. The set~$N(a)$ is in bijection with the set of states, and since these are labeled~$d_1$, $\ldots $, $d_6$, we have a natural order. The adjacency matrix is then 
\[A = \left(\begin{array}{rrrrrr}
2  & 0 & 3  & 0 & 2  & 0 \\
0 & 2  & 2  & 0 & 0 & 3  \\
2  & 3  & 2  & 0 & 0 & 0 \\
0 & 0 & 0 & 2  & 3  & 2  \\
3  & 0 & 0 & 2  & 2  & 0 \\
0 & 2  & 0 & 3  & 0 & 2 
\end{array}\right) \, .  \]
To construct the vector~$T$, for each integer~$i$ we compute the indices~$j_1 = t_1(i), \ldots, j_6 = t_6(i)$, and the~$i$-th line of~$T$ is~$d_{j_1} + \cdots + d_{j_6}$. Explicitly 
\[ T = \left(\begin{array}{c}
d_{1}  + 3 d_{3}  + 2 d_{5}  \\
d_{2}  + 2 d_{3}  + 3 d_{6}  \\
2 d_{1}  + 3 d_{2}  + d_{3}  \\
d_{4}  + 3 d_{5}  + 2 d_{6}  \\
3 d_{1}  + 2 d_{4}  + d_{5}  \\
2 d_{2}  + 3 d_{4}  + d_{6} 
\end{array}\right) \, .  \]
The fraction~$L(b, x)$ is then the first line of~$(I - xA)^{-1} xT$. This straightforward task was performed by a computer, and the result is that which we have given in the introduction.
\end{ex}

\subsection{Applications to the frequency of letters}

The fraction~$L(a, x)$ is useful in providing information on the frequency of letters within~$a$. Recall that we write~$a[n, d_i]$ for the number of occurences of~$d_i$ among the first~$n$ terms of~$a$. The next lemma provides the connection, and its proof is immediate.

\begin{lem}
The expansion of~$L(a, x)$ is 
\[ L(a, x) = \sum_{n \ge 1} (m_{1, n} d_1 + \cdots + m_{r, n} d_r) x^n \]
where~$m_{i,n}$ is the number of integers~$m$, having length~$n$ when written in base~$p$, such that~$a_m= d_i$. As a result, if 
\[ L(a, x) = \frac{d_1 P_1 + \cdots + d_r P_r} {D}  \]
with~$P_i, D \in \z[x]$, and if we write 
\[ \frac{P_i} {D} = \sum_{n \ge 1} s_n x^n \, ,   \]
then 
\[ \sum_{j=1}^n \, s_j = a[p^n, d_i] \, . \qquad\square \]

\end{lem}

We now make a few technical points.

\begin{prop}
Let~$D \in \C[x]$. Assume that the roots of~$D$ have absolute value~$\ge \frac{1} {p}$, and let the roots of absolute value~$\frac{1} {p}$ be written $\alpha _k = \frac{1} {p} e^{i \theta_k}$ for~$k= 1, 2, \ldots $ Assume further that each~$\alpha_k$ is a simple root, and that~$D(0) \ne 0$. Finally, let~$N \in \C[x]$ be any polynomial.

If we write 
\[ \frac{N} {D} = \sum_{n \ge 0} \, s_n x^n \, ,   \]
then we have the following estimate:
\[ \frac{\sum_{j=1}^n s_j} {p^n} = \sum_{k} \frac{N(\alpha_k) \operatorname{Res}(\frac{1} {D}, \alpha_k)} {\alpha_k^2 - \alpha_k}   e^{-ni \theta_k} + o(1) \, .  \]


\end{prop}

\begin{proof}
We examine a few particular cases for~$D$ and~$N$. Suppose first that~$N=1$ and
\[ \frac{1} {D} = \frac{1} {(x-\alpha )^m} = \frac{c} {(1 - \gamma x)^m}  \]
where~$\gamma = \frac{1} {\alpha }$ and~$c = \frac{1} {\alpha^m}$ is constant. Here we assume that~$|\alpha | > \frac{1} {p}$ so~$|\gamma | < p$. In this case~$s_n = P(n) \gamma ^n$ where~$P$ is a polynomial of degree~$< m$. Choose a constant~$C >0$ such that~$|P(n)| \le C n^{m-1}$ for all~$n \ge 1$. Thus 
\[ \left| \frac{\sum_{j=1}^n s_j} {p^n} \right| \le \frac{C n^m |\gamma | (1 + |\gamma | + \cdots |\gamma |^{n-1})} {p^n} = \frac{C |\gamma | n^m} {|\gamma | -1} \cdot \frac{|\gamma |^n -1} {p^n} = o(1) \, .   \]

Now suppose~$\alpha = \frac{1} {p} e^{i \theta }$, that~$N=1$ and that 
\[ \frac{1} {D} = \frac{1} {X - \alpha } = - \frac{\gamma } {1 - \gamma x}  \]
where~$\gamma = \frac{1} {\alpha }$, so that~$s_n = -\gamma ^{n+1}$. Now 
\begin{align*}
\frac{\sum_{j=1}^n s_j} {p^n} & = \frac{\gamma^2} {\gamma - 1} \cdot \frac{\gamma^n - 1} {p^n} \\
                        & = \frac{\gamma ^2} {1 - \gamma } (e^{-ni \theta  } + o(1))  \\
                        & = \frac{1} {\alpha^2 - \alpha } e^{-ni \theta } + o(1) \, . 
\end{align*}
The general case is obtained by writing the partial fraction decomposition of~$\frac{N} {D}$. 
\end{proof}

The Proposition and the Lemma together show:

\begin{thm}
Suppose that~$a$ is~$p$-automatic, on the alphabet~$\{ d_1, \ldots, d_r \}$, and that 
\[ L(a, x) = \frac{d_1 P_1 + \cdots + d_r P_r} {D}  \]
with~$P_i, D \in \z[x]$. Assume that the (complex) roots of~$D$ have absolute value~$\ge \frac{1} {p}$, and let the roots of absolute value~$\frac{1} {p}$ be written $\alpha _k = \frac{1} {p} e^{i \theta_k}$ for~$k= 1, 2, \ldots $ Assume further that each~$\alpha_k$ is a simple root.

Then one has 
\[ \frac{a[p^n, d_i]} {p^n} = \sum_{k} \frac{P_i(\frac{1} {p})\operatorname{Res}(\frac{1} {D}, \alpha_k)} {\alpha_k^2 - \alpha_k}   e^{-ni \theta_k} + o(1) \, .  \]
\end{thm}

\begin{coro}
Suppose that the only root of~$D$ of absolute value~$\frac{1} {p}$ is~$\frac{1} {p}$ (and that this root is simple). Then 
\begin{align*}
\lim_{n \to +\infty} \frac{a[p^n, d_i]} {p^n} & = \frac{p^2} {1-p} \, P_i\left( \frac{1} {p} \right)  \operatorname{Res}\left( \frac{1} {D}, \frac{1} {p} \right) \\
             & = \frac{P_i(\frac{1} {p})} {\sum_j P_j(\frac{1} {p})} \, . 
\end{align*}

\end{coro}

\begin{proof}
The first equality follows directly from the Theorem, while the second is drawn by observing that the various limits, as~$i$ ranges from~$1$ to~$r$, all exist and sum up to~$1$.
\end{proof}

\begin{ex}
As explained in the introduction, the corollary applies to the sequence~$b$ of Apéry numbers mod~$7$, produced by the automaton on Figure 2. The polynomials~$P_i$, different as they are, give the same value at~$\frac{1} {7}$. It follows that 
\[ \lim_{n\to +\infty} \frac{b[7^n, d_i]} {7^n} = \frac{1} {6} \, .   \]
\end{ex}

\subsection{A complete example}

Consider the~$2$-automaton below.

\begin{center}
\includegraphics[width=.7\textwidth]{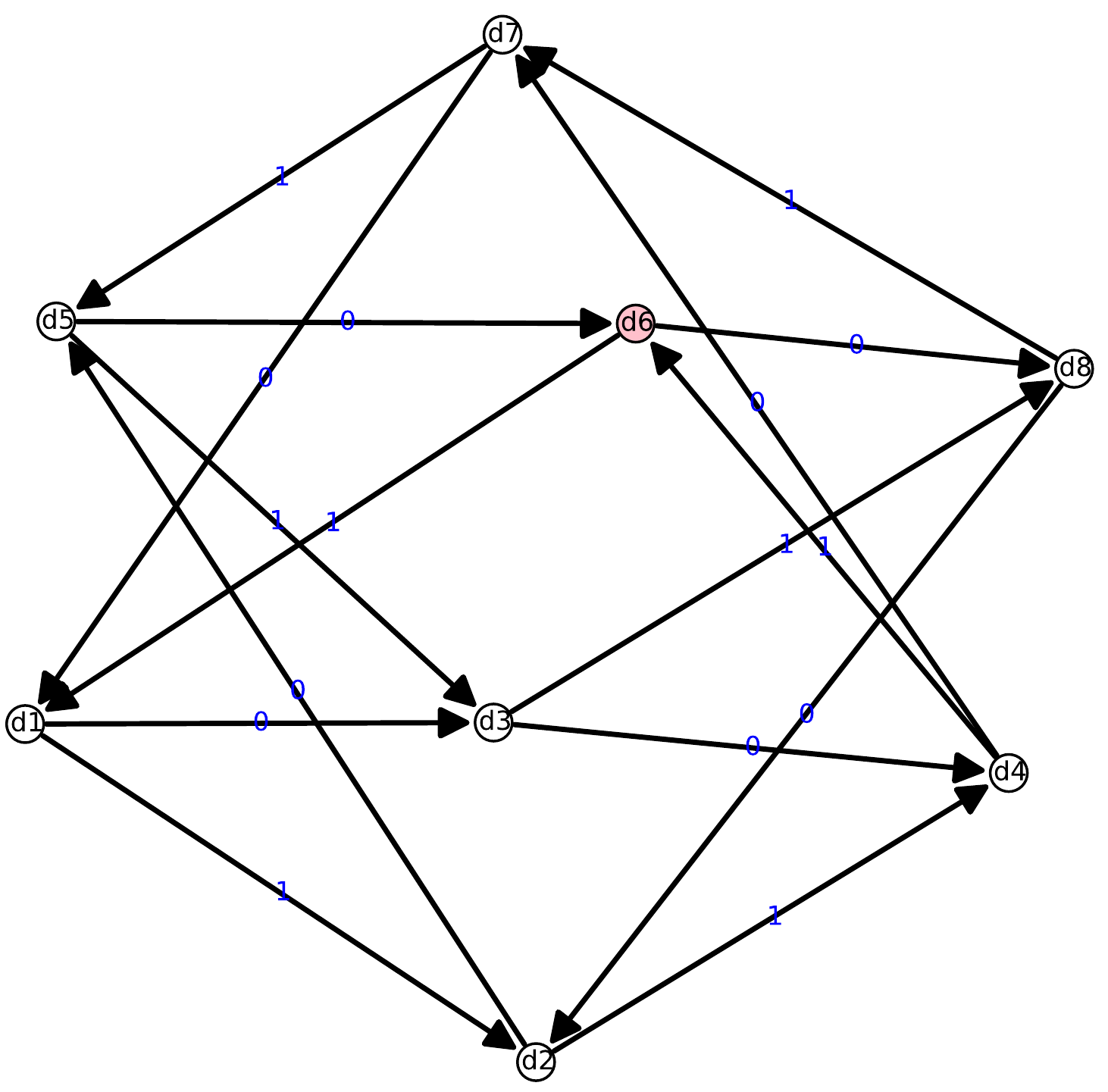}

{\em \footnotesize Fig.\ 3}
\end{center}

We shall study the corresponding sequence~$a$, and indulge in all the details. 

The first thing to notice is that each state has two incoming arrows with labels~$0$ and~$1$, and also two outgoing arrows with labels~$0$ and~$1$, so we are in the presence of a Schreier graph. To check that it is in fact a Cayley graph, we have no choice but consider the group~$G$ generated by the permutations 
\[ t_0= (1,3,4,7)(5,6,8,2) \quad\textnormal{and}\quad t_1=(1,2,4,6)(5,3,8,7) \, .   \]
One checks that~$G$ has order~$8$, and indeed is isomorphic to the group of quaternions (for example by asking a computer). Lemma~\ref{lem-schreier-is-cayley} guarantees that we have a Cayley graph. The labeling map~$\tau \colon G \to \{ d_1, \ldots, d_8 \}$ is injective, so the alphabet is essentially~$G$ itself, and we have a self-similar sequence. By Theorem~\ref{thm-self-similar-all-p}, the graph~$\Gamma (a)$ is simply that on Figure 3.

To proceed with the computations of~$L(a, x)$, we order the sequences in~$N(a)$ according to the labels on their initial states, so that the first sequence is that with initial state~$d_1$, the second is that with initial state~$d_2$, and so on. The sequence~$a$ itself, according to the picture, has initial state~$d_6$, so it is the sixth in this order.

The adjacency matrix is then
\[ A= \left(\begin{array}{rrrrrrrr}
0 & 1 & 1 & 0 & 0 & 0 & 0 & 0 \\
0 & 0 & 0 & 1 & 1 & 0 & 0 & 0 \\
0 & 0 & 0 & 1 & 0 & 0 & 0 & 1 \\
0 & 0 & 0 & 0 & 0 & 1 & 1 & 0 \\
0 & 0 & 1 & 0 & 0 & 1 & 0 & 0 \\
1 & 0 & 0 & 0 & 0 & 0 & 0 & 1 \\
1 & 0 & 0 & 0 & 1 & 0 & 0 & 0 \\
0 & 1 & 0 & 0 & 0 & 0 & 1 & 0
\end{array}\right) \, ,  \]
and the vector~$T$ is: 
\[ T= \left(\begin{array}{c}
d_2 \\ d_4 \\ d_8 \\ d_6 \\ d_3 \\ d_1 \\ d_5 \\ d_7
\end{array}\right)  \, . \]
(because the first sequence starts with~$d_2$, the second with~$d_4$, etc). The fraction~$L(a, x)$ is the~$6$-th line in the vector 
\[ (I - xA)^{-1} x T \, ,   \]
and we find 
\[ L(a, x) = \frac{d_1 P_1 + \cdots + d_8 P_8} {8 x^{4} + 2 x^{2} - 1} \]
with 
\[ P_1 =  2 x^{5} + 2 x^{3} -  x \, , \qquad P_2 = x^{4} -  x^{2} \, ,  \]
\[ P_3= -3 x^{4} \, , \qquad P_4= 2 x^{5} - 2 x^{3} \, , \]
\[ P_5= -2x^5 - x^{3} \, , \qquad P_6= -3 x^{4} \, , \]
\[ P_7= x^{4} -  x^{2} \, , \qquad P_8 = -2x^5 - x^{3} \, .   \]
So the denominator is 
\[ D= (2 x - 1) \cdot (2 x + 1) \cdot (2 x^{2} + 1) \, ,  \]
its roots being~$\frac{1} {2}$, $-\frac{1} {2}$, $i \frac{\sqrt 2} {2}$ and~$-i \frac{\sqrt 2} {2}$, the last two having modulus~$\frac{\sqrt 2} {2} > \frac{1} {2}$. The residue of $\frac{1} {D}$ at~$\pm \frac{1} {2}$ is~$\pm \frac{1} {6}$.

Thus we have 
\[ \frac{a[2^n, d_i]} {2^n} = - \left[  \frac{2} {3} P_i\left(\frac{1} {2}\right) + (-1)^n \frac{2} {9} P_i\left(-\frac{1} {2}\right)  \right] + o(1) \, . \]
For~$i \in \{ 1, 4, 5, 8 \}$, we have~$P_i(\frac{1} {2}) = - \frac{3} {16}$ and~$P_i(-\frac{1} {2}) = \frac{3} {16}$ so 
\[ \frac{a[2^n, d_i]} {2^n} = \frac{1} {8} + (-1)^{n+1} \frac{1} {24}  + o(1) \, .  \]
For these indices, we conclude that when~$n$ is large and even, the ratio~$a[2^n, d_i]/2^n$ is close to~$\frac{1} {8} - \frac{1} {24} = \frac{1} {12}$, but when~$n$ is odd the ratio is close to~$\frac{1} {8} + \frac{1} {24} = \frac{1} {6}$.

For~$i \in \{ 2,3,6,7 \}$, on the other hand, we find~$P_i(\frac{1} {2}) = P_i(-\frac{1} {2}) = -\frac{3} {16}$, and so 
\[ \frac{a[2^n, d_i]} {2^n} =  \frac{1} {8} + (-1)^n \frac{1} {24}  + o(1)  \, . \]
The ratio in this case is close to~$\frac{1} {6}$ when~$n$ is even and large, and close to~$\frac{1} {12}$ if~$n$ is odd and large.

In particular, none of the ratios~$\frac{a[2^n, d_i]} {2^n}$ converges. Also note that when~$n$ is even, approximately two thirds of the first~$2^n$ terms of the sequence~$a$ are in the set~$\{ d_2, d_3, d_6, d_8 \}$. When~$n$ is odd, exactly the opposite is true.

\bibliography{myrefs}
\bibliographystyle{amsalpha}

\end{document}